\documentclass[11pt,leqno]{amsart}

\usepackage{amsmath,amssymb,amsthm}
\usepackage[a4paper,margin=2.5cm]{geometry}
\usepackage[hidelinks]{hyperref}

\newtheorem{theorem}{Theorem}[section]
\newtheorem*{theorem*}{Theorem}
\newtheorem{proposition}[theorem]{Proposition}
\newtheorem{corollary}[theorem]{Corollary}
\newtheorem*{questionA}{Question A}
\newtheorem*{questionB}{Question B}
\theoremstyle{definition}
\newtheorem{definition}[theorem]{Definition}
\newtheorem{example}[theorem]{Example}
\theoremstyle{remark}
\newtheorem{remark}[theorem]{Remark}

\DeclareMathOperator{\Id}{Id}
\newcommand{\R}{\mathbb R}

\title{Characterizing the Daugavet property by squares of rank-one operators}
\author{Johann Langemets}
\address{Institute of Mathematics and Statistics, University of Tartu,
Narva mnt 18, 51009 Tartu, Estonia}
\email{johann.langemets@ut.ee}
\urladdr{https://www.johannlangemets.com/}
\thanks{This work was supported by the Estonian Research Council grant (PRG2545) and partially supported by the grant PID2025-167660NB-I00 funded by MICIU/AEI/10.13039/501100011033 and ERDF/EU}
\date{}

\begin{document}

\begin{abstract}
We prove that, for real Banach spaces of dimension greater than one, each of the identities $\|\Id+T^2\|=1+\|T^2\|$ and $\|\Id-T^2\|=1+\|T^2\|$, required for every rank-one operator $T$, characterizes the Daugavet property, thereby answering a question posed by Kadets, Mart\'in, and Mer\'i (2007). Consequently, every extremely non-complex Banach space of dimension greater than one has the Daugavet property, answering a question of Mart\'in and Mer\'i (2011). We also compare pointwise versions of the square Daugavet properties with Daugavet points, $\Delta$-points, and $\nabla$-points.
\end{abstract}

\maketitle

\section{Introduction}

Throughout the paper, Banach spaces are nonzero and real. A Banach
space $X$ has the \emph{Daugavet property} if
\[
  \|\Id+T\|=1+\|T\|
\]
for every rank-one operator $T\colon X\to X$. The equation goes back
to Daugavet~\cite{Daugavet1963}, who proved it for compact operators
on $C[0,1]$. Kadets, Shvidkoy, Sirotkin, and Werner~\cite{KSSW2000}
studied the property systematically and showed that the rank-one
condition implies the same identity for every weakly compact
operator. Classical examples include $C(K)$ for perfect compact
Hausdorff spaces $K$, and $L_1(\mu)$ and $L_\infty(\mu)$ for
non-atomic measures $\mu$. A space with the Daugavet property
contains an isomorphic copy of $\ell_1$, does not embed into a
space with an unconditional basis, and every slice of its unit ball has
diameter two; see~\cite{KSSW2000}. For a broader account, see the forthcoming
monograph~\cite{KMRZW}.

For $\sigma\in\{-1,1\}$, following Oikhberg~\cite{Oikhberg2007}, we say that $X$ has the
\emph{$\sigma$-square Daugavet property} if
\begin{equation}\label{eq:signed-square}
  \|\Id+\sigma T^2\|=1+\|T^2\|
\end{equation}
for every rank-one operator $T\colon X\to X$. The cases $\sigma=1$
and $\sigma=-1$ are called the \emph{positive} and \emph{negative
square Daugavet properties}, respectively. These identities were first
considered by Kadets, Mart\'in, and Mer\'i~\cite{KMM2007} in their
study of norm equalities for functions of rank-one operators.
More generally, they studied equations of the form
\begin{equation*}
  \|\Id+\psi(T)\|=\varphi(\|\psi(T)\|),
\end{equation*}
where $\psi$ is an entire function on the scalar field and
$\varphi\colon\R_{\ge0}\to\R_{\ge0}$ is continuous. They obtained
rigidity results under suitable hypotheses on these functions, but
left open whether the two signed square identities, corresponding
to $\psi(z)=\sigma z^2$ and $\varphi(t)=1+t$, force the Daugavet
property. With the necessary one-dimensional exception removed,
their question can be stated as follows.

\begin{questionA}[{\cite{KMM2007}}]
Let $X$ be a real Banach space with $\dim X>1$. Does either the
positive or the negative square Daugavet property imply the
Daugavet property?
\end{questionA}

Each signed square Daugavet property implies the \emph{alternative Daugavet property}, introduced by Mart\'in and Oikhberg~\cite{MO2004}, which requires
\[
  \max\{\|\Id+T\|,\|\Id-T\|\}=1+\|T\|
\]
for every rank-one operator $T$. Indeed, writing $T=f\otimes x$, one of $T$ and $-T$ satisfies the sign condition in Proposition~\ref{prop:kmm-signed} below, and hence the Daugavet equation. The alternative Daugavet property is strictly weaker than the Daugavet property~\cite{MO2004}; see~\cite{LLMPRZ2026} for a recent strengthening and further references. Question~A asks whether either signed square hypothesis yields the stronger conclusion that $X$ has the Daugavet property.

Earlier work established geometric consequences of these
identities. For spaces of dimension greater than one,
Oikhberg~\cite[Theorem~1]{Oikhberg2007} showed that each signed square
Daugavet property excludes strongly exposed points of the unit
ball, and that every slice of the unit ball has diameter two in the case of the negative square Daugavet property. Mart\'in and Mer\'i~\cite[Theorem~2]{MM2011} subsequently
obtained a uniform positive lower bound for the diameters of
slices of the unit ball under the positive square Daugavet property.

If the positive identity in~\eqref{eq:signed-square} holds for
\emph{every bounded linear operator} $T\colon X\to X$, the space $X$ is
called \emph{extremely non-complex}. Whether such spaces other
than $\R$ exist was asked in~\cite{KMM2007}. Koszmider, Mart\'in,
and Mer\'i~\cite{KMM2009} constructed infinite-dimensional
examples, and subsequently studied isometries on such
spaces~\cite{KMM2011}. Mart\'in and Mer\'i~\cite{MM2011} later
asked whether every extremely non-complex space has the Daugavet
property.

\begin{questionB}[{\cite{MM2011}, \cite[Question~12.3]{KMRZW}}]
Does every extremely non-complex space of dimension greater than one have the Daugavet property?
\end{questionB}

We answer both questions affirmatively. The main result is the following equivalence.

\begin{theorem*}
Let $X$ be a real Banach space with $\dim X>1$.
The following are equivalent:
\begin{itemize}
\item[(i)] $X$ has the Daugavet property;
\item[(ii)] $X$ has the positive square Daugavet property;
\item[(iii)] $X$ has the negative square Daugavet property.
\end{itemize}
In particular, every extremely non-complex Banach space of
dimension greater than one has the Daugavet property.
\end{theorem*}

We prove the main theorem by treating the positive
and negative square Daugavet properties separately (see Theorem~\ref{thm:positive-global} and Corollary~\ref{cor:global}). In Section~\ref{sec:comparison}, we compare positive square points with the
$\Delta$- and $\nabla$-points studied in~\cite{AHLP2020,HLPV2024}
and give examples establishing the strictness and incomparability
of the remaining conditions.

\section{Main results}

We write $B_X$, $S_X$, and $X^*$ for the closed unit ball, the unit
sphere, and the dual of a Banach space $X$, respectively. For
$f\in X^*$ and $x\in X$, the operator $f\otimes x$ is defined by
\[
  (f\otimes x)(z)=f(z)x\qquad(z\in X).
\]

We call a pair $(x,f)\in S_X\times S_{X^*}$ \emph{good} if, for
every $\varepsilon>0$, there is $y\in S_X$ such that
\begin{equation}\tag{$G$}\label{eq:good}
  f(y)>1-\varepsilon
  \qquad\text{and}\qquad
  \|x+y\|>2-\varepsilon.
\end{equation}
The Daugavet property is equivalent to requiring that every pair $(x,f)\in S_X\times S_{X^*}$ is good~\cite[Lemma~2.2]{KSSW2000}. For the signed square Daugavet
properties, we use the following characterization of Kadets,
Mart\'in, and Mer\'i~\cite{KMM2007}.

\begin{proposition}[{\cite[Proposition~4.9]{KMM2007}}]\label{prop:kmm-signed}
Let $X$ be a Banach space and let $\sigma\in\{-1,1\}$. The following assertions are equivalent:
\begin{itemize}
\item[(i)] $X$ has the $\sigma$-square Daugavet property;
\item[(ii)] $\|\Id+f\otimes x\|=1+\|f\otimes x\|$ for all
$f\in X^*$ and $x\in X$ with $\sigma f(x)\ge0$;
\item[(iii)] the pair $(x,f)$ is good for every $x\in S_X$ and
$f\in S_{X^*}$ with $\sigma f(x)\ge0$.
\end{itemize}
\end{proposition}

\subsection{The positive square property}

We begin by answering Question~A for the positive square Daugavet property.

\begin{theorem}\label{thm:positive-global}
Let $X$ be a Banach space with $\dim X>1$. Then $X$ has the
positive square Daugavet property if and only if it has the
Daugavet property.
\end{theorem}

\begin{proof}
Clearly, the Daugavet property implies the positive square Daugavet property. For the converse, assume that $X$ has the positive square Daugavet
property. We will show that every pair
$(x,f)\in S_X\times S_{X^*}$ is good.

Fix $f\in S_{X^*}$, and put
\[
  A_f=\{z\in S_X:(z,f)\text{ is good}\}.
\]
By Proposition~\ref{prop:kmm-signed}, this set contains every
$z\in S_X$ with $f(z)>0$, so it is nonempty. We claim that
\begin{equation}\tag{$\bigstar$}\label{eq:propagation}
  z\in A_f,\quad x\in S_X,\quad \|x-z\|<\tfrac12
  \quad\Longrightarrow\quad x\in A_f.
\end{equation}

To prove~\eqref{eq:propagation}, let $z\in A_f$ and $x\in S_X$ with
$\|x-z\|<\tfrac12$. Write $\rho=\|x-z\|$ and $\delta=1-2\rho>0$. Fix
$0<\varepsilon<1$ and set
$\eta=\tfrac1{12}\min\{\varepsilon,\delta\}$.
Since $z\in A_f$, there is $y\in S_X$ with
\[
  f(y)>1-\eta
  \qquad \text{and} \qquad
  \|z+y\|>2-\eta.
\]
Choose $g\in S_{X^*}$ with $g(z+y)=\|z+y\|$. As $g(z),g(y)\le1$,
we have
\[
  g(z)>1-\eta
   \qquad \text{and} \qquad
  g(y)>1-\eta.
\]

Set $F=f+2g$. Then
\[
  \|F\|\ge F(y)=f(y)+2g(y)>3(1-\eta).
\]
On the other hand, $f(x)\ge-1$ and
$g(x)\ge g(z)-\|x-z\|>1-\eta-\rho$, so
\[
  F(x)>-1+2(1-\eta-\rho)
       =\delta-2\eta\ge\tfrac56\delta>0.
\]
Thus $h=F/\|F\|\in S_{X^*}$ satisfies $h(x)>0$.
By Proposition~\ref{prop:kmm-signed}, the pair $(x,h)$ is good.
There is therefore $w\in S_X$ with
\[
  h(w)>1-\eta
   \qquad \text{and} \qquad
  \|x+w\|>2-\eta.
\]
We obtain $F(w)>(1-\eta)\|F\|>3(1-\eta)^2$.
Since $g(w)\le1$,
\begin{align*}
  f(w)&=F(w)-2g(w)>3(1-\eta)^2-2\\
      &\ge1-6\eta\ge1-\tfrac\varepsilon2>1-\varepsilon.
\end{align*}
Together with $\|x+w\|>2-\eta>2-\varepsilon$, this proves that
$(x,f)$ is good. Hence~\eqref{eq:propagation} holds.

Property~\eqref{eq:propagation} shows that $A_f$ is relatively open
in $S_X$. Its complement is relatively open as well: if
$x\notin A_f$, no point $z\in A_f$ can satisfy $\|x-z\|<1/2$,
since otherwise~\eqref{eq:propagation} would imply $x\in A_f$. Hence $A_f$ is a clopen subset of $S_X$.
The unit sphere of a real normed space of dimension greater than
one is connected. Since $A_f$ is nonempty, it follows that
$A_f=S_X$. As $f$ was arbitrary, every normalized pair is good,
and $X$ has the Daugavet property.
\end{proof}

\begin{remark}\label{rem:real-line}
The assumption $\dim X>1$ is necessary. On the real line, every
operator has the form $T=t\Id$, and $|1+t^2|=1+t^2$. Thus $\R$ is an extremely non-complex space, and hence it
has the positive square Daugavet property. It does not have the
Daugavet property, as witnessed by the rank-one operator $T=-\Id$.
\end{remark}

Since rank-one operators are a particular case of bounded linear operators, Theorem~\ref{thm:positive-global} also answers Question~B
raised in~\cite{MM2011}.

\begin{corollary}\label{cor:extremely-non-complex}
Every extremely non-complex Banach space of dimension greater
than one has the Daugavet property.
\end{corollary}

\subsection{The negative square property}

In~\cite{AHLP2020}, Abrahamsen, Haller, Lima, and Pirk introduced the notion of a Daugavet point, which is a pointwise version of the Daugavet property. An element $x\in S_X$ of a Banach space $X$ is a \emph{Daugavet point} if $(x,f)$ is
good for every $f\in S_{X^*}$. We now introduce the corresponding pointwise versions of the signed square properties, of which the negative case will turn out to be equivalent to the Daugavet point condition.

\begin{definition}\label{def:square-points}
Let $X$ be a Banach space and let $x\in S_X$. We say that $x$ is
\begin{itemize}
\item a \emph{positive square point} if $(x,f)$ is good for every
$f\in S_{X^*}$ with $f(x)\ge0$;
\item a \emph{negative square point} if $(x,f)$ is good for every
$f\in S_{X^*}$ with $f(x)\le0$.
\end{itemize}
\end{definition}

Proposition~\ref{prop:kmm-signed} says that $X$ has the positive
(respectively, negative) square Daugavet property if and only if
every point of $S_X$ is a positive (respectively, negative) square
point. Every Daugavet point satisfies both pointwise conditions.
The next result proves the converse for negative square points.
For positive square points, the converse fails, as
Example~\ref{ex:linfty} below shows.

\begin{theorem}\label{thm:negative-point} Let $X$ be a Banach space. Then an element $x\in S_X$ is a negative square point if and only if
it is a Daugavet point.
\end{theorem}

\begin{proof}
Clearly, every Daugavet point is a negative square point. For the converse,
suppose that $x$ is a negative square point. Fix $f\in S_{X^*}$ and
$0<\varepsilon<1$. We will verify~\eqref{eq:good}. There is nothing
to prove if $f(x)\le0$, so assume that $f(x)>0$ and put
$\delta=\varepsilon/8$.

Apply the negative square condition to $-f$ and negate the resulting
witness. This gives $y\in S_X$ with
\[
  f(y)>1-\delta \qquad \text{and} \qquad
  \|y-x\|>2-\delta.
\]
Choose $g\in S_{X^*}$ with $g(y-x)=\|y-x\|$. Then
$g(y)>1-\delta$ and $g(x)<-1+\delta$. For $F=f+2g$, we obtain
\[
  F(x)<-1+2\delta<0 \qquad \text{and} \qquad
  \|F\|\ge F(y)>3(1-\delta).
\]
Apply the negative square condition to $F/\|F\|$, obtaining
$w\in S_X$ with
\[
  \|x+w\|>2-\delta \qquad \text{and} \qquad
  F(w)>(1-\delta)\|F\|>3(1-\delta)^2.
\]
Consequently,
\[
  f(w)=F(w)-2g(w)>3(1-\delta)^2-2
       >1-6\delta>1-\varepsilon.
\]
Since also $\|x+w\|>2-\delta>2-\varepsilon$, the pair $(x,f)$ is good.
As $f$ was arbitrary, $x$ is a Daugavet point.
\end{proof}

Proposition~\ref{prop:kmm-signed} and
Theorem~\ref{thm:negative-point} together give an affirmative answer to the negative case of Question~A without a dimension restriction.

\begin{corollary}\label{cor:global}
A Banach space has the negative square Daugavet property if and
only if it has the Daugavet property.
\end{corollary}

\section{Comparison with diametral points}\label{sec:comparison}

Theorem~\ref{thm:negative-point} identifies negative square points
with Daugavet points. No such identification is available in the
positive case, and we now locate positive square points among the
diametral notions studied alongside Daugavet points. Abrahamsen,
Haller, Lima, and Pirk~\cite{AHLP2020} introduced $\Delta$-points
together with Daugavet points, and Haller, Langemets, Perreau, and
Veeorg~\cite{HLPV2024} later introduced $\nabla$-points. Requiring every point of the unit sphere to be a $\Delta$-point gives the \emph{diametral local diameter two property}~\cite{AHLP2020}. When $\dim X>1$, requiring every point of the unit sphere to be a $\nabla$-point gives the Daugavet property~\cite[Theorem~1.6]{HLPV2024}. For a recent survey of these notions and their interrelations, see~\cite{MPRZ2024}.

For $f\in S_{X^*}$ and $\alpha>0$, we write
\[
  S(f,\alpha)=\{y\in B_X:f(y)>1-\alpha\}
\]
for the corresponding slice of $B_X$. An equivalent formulation of the notion of a Daugavet point in terms of slices is that $x\in S_X$ is a Daugavet point if and only if $\sup_{y\in S}\|x-y\|=2$ for every slice $S$ of $B_X$~\cite{AHLP2020}.

\begin{definition}[{\cite{AHLP2020,HLPV2024}}]
\label{def:diametral-points}
Let $X$ be a Banach space and let $x\in S_X$. Then $x$ is a
\begin{itemize}
\item \emph{$\Delta$-point} if
$\sup_{y\in S}\|x-y\|=2$ for every slice $S$ of $B_X$ containing
$x$;
\item \emph{$\nabla$-point} if
$\sup_{y\in S}\|x-y\|=2$ for every slice $S$ of $B_X$ not
containing $x$.
\end{itemize}
\end{definition}

From this slice characterization, it is clear that $x\in S_X$ is a Daugavet point if and only if it is simultaneously a $\Delta$-point and a $\nabla$-point. It turns out that every $\nabla$-point is a positive square point, as we now show.

\begin{proposition}\label{prop:nabla-positive}
Let $X$ be a Banach space. If $x$ is a $\nabla$-point, then $x$ is a positive square point.
\end{proposition}

\begin{proof}
Let $x$ be a $\nabla$-point and let $f\in S_{X^*}$ satisfy
$f(x)\ge0$. Fix $0<\varepsilon<1$ and put $\delta=\varepsilon/2$.
Since $(-f)(x)\le0<1-\delta$, the slice $S(-f,\delta)$ does not
contain $x$. There is therefore $z\in S(-f,\delta)$ with
$\|x-z\|>2-\delta$. Set $y=-z$. Then
\[
  f(y)>1-\delta
  \qquad\text{and}\qquad
  \|x+y\|>2-\delta.
\]
Therefore, we have $1-\delta<\|y\|\le1$.
For $w=y/\|y\|\in S_X$, it follows that
\[
  f(w)\ge f(y)>1-\delta>1-\varepsilon
\]
and
\[
  \|x+w\|\ge\|x+y\|-\|w-y\|
  >2-\delta-(1-\|y\|)>2-\varepsilon.
\]
Thus $(x,f)$ is good, and $x$ is a positive square point.
\end{proof}

\begin{remark}\label{rem:nabla-daugavet}
Proposition~\ref{prop:nabla-positive} and
Theorem~\ref{thm:positive-global} give a new proof of the result
of~\cite[Theorem~1.6]{HLPV2024} mentioned above. Indeed, let $X$ be a Banach
space with $\dim X>1$ in which every point of $S_X$ is a
$\nabla$-point. By Proposition~\ref{prop:nabla-positive}, every
point of $S_X$ is a positive square point, so $X$ has the positive
square Daugavet property by Proposition~\ref{prop:kmm-signed}.
Theorem~\ref{thm:positive-global} now shows that $X$ has the
Daugavet property.
\end{remark}

Hence we have the following chain of implications for these notions.
\[
\begin{gathered}
 \Delta\text{-point}\overset{\text{(a)}}{\Longleftarrow} \text{Daugavet point}
  \ \overset{\text{(b)}}{\Longrightarrow}\ \nabla\text{-point}
  \overset{\text{(c)}}{\Longrightarrow}\ \text{positive square point}.
\end{gathered}
\]
The implications (a) and (b) are strict, as shown in~\cite[Example~4.7]{AHLP2020} and~\cite[Example~1.5]{HLPV2024}, respectively. The next two examples show that implication (c) is also strict and that the positive square point and $\Delta$-point notions are incomparable.

\begin{example}\label{ex:linfty}
In $X=\ell_\infty^2$, the point $x=(1,0)$ is a positive square
point but is neither a $\nabla$-point nor a $\Delta$-point.
\end{example}

\begin{proof}
Let $f=(a,b)\in S_{\ell_1^2}$ with $f(x)=a\ge0$. Choose
$y=(1,\operatorname{sgn}b)$, taking either sign when $b=0$.
Then $f(y)=a+|b|=1$ and $\|x+y\|_\infty=2$, so $x$ is a
positive square point.

For $g=(1/2,1/2)$, the slice $S(g,1/4)$ excludes $x$.
If $z=(s,t)$ belongs to this slice, then $s+t>3/2$, so
$s,t>1/2$. Therefore
\[
  \|x-z\|_\infty=\max\{|1-s|,|t|\}\le1.
\]
Thus $x$ is not a $\nabla$-point.

For $h=(1,0)$, the slice $S(h,1/2)$ contains $x$, and every point of this slice also has distance
at most one from $x$. Hence $x$ is not a $\Delta$-point.
\end{proof}

\begin{example}\label{ex:delta-not-positive-square}
Let $c$ be the space of convergent sequences with the supremum
norm. In $X=c\oplus_2\R$, the point $x=(\mathbf1,0)$ is a
$\Delta$-point but is not a positive square point.
\end{example}

\begin{proof}
Let $e_n$ be the coordinate vectors in $c$. The vectors
\[
  y_n=(\mathbf1-2e_n,0)\in S_X
\]
converge weakly to $x$ and satisfy $\|x-y_n\|=2$. Since slices are relatively weakly open, every slice containing $x$ contains $y_n$ for all sufficiently large
$n$, proving that $x$ is a $\Delta$-point.

The functional $f(v,t)=t$ has norm one and $f(x)=0\ge0$, so $f$ is
an admissible test functional for the positive square point
condition. We show that the pair $(x,f)$ is not good for
$\varepsilon=1/8$: no $y=(v,t)\in S_X$ satisfies both
$f(y)>1-\varepsilon$ and $\|x+y\|>2-\varepsilon$.

Suppose $y=(v,t)\in S_X$ satisfies $t=f(y)>1-\varepsilon=7/8$. Since
$\|v\|_\infty^2+t^2=1$,
\[
  \|v\|_\infty^2=1-t^2=(1-t)(1+t)\le2(1-t)<2\varepsilon=\tfrac14,
\]
so $\|v\|_\infty<\tfrac12$. Hence
\[
  \|x+y\|^2=\|\mathbf1+v\|_\infty^2+t^2
  \le(1+\|v\|_\infty)^2+t^2=2+2\|v\|_\infty<3,
\]
so $\|x+y\|<\sqrt{3}<\tfrac{15}8=2-\varepsilon$.

Thus $(x,f)$ is not good, and $x$ is not a positive square point.
\end{proof}

\section*{AI disclosure statement}
OpenAI ChatGPT, accessed through the ChatGPT for Academic Researchers program, was used solely for literature search and to improve the clarity, grammar, and style of the manuscript. All suggested changes were reviewed and approved by the author.

\end{document}